\documentclass[12pt]{article}

\usepackage{tikz}

\usepackage{tikz-cd}

\usepackage{authblk}

\def\epsilon{\varepsilon}
\def\bbm{\begin{bmatrix}}
\def\ebm{\end{bmatrix}}

\usepackage{MathSettingsBen}
\usepackage{quiver}

\newcommand{\bn}{{\underline{n}}}
\newcommand{\bl}{{\underline{l}}}
\newcommand{\bk}{{\underline{k}}}

\newcommand{\FinSet}{\cat{FinSet}}
\renewcommand{\M}{\Pi}
\newcommand{\Md}{{\M^{\op}}}
\newcommand{\g}{\gamma}
\newcommand{\G}{\Gamma}
\newcommand{\Gd}{{\G^{\op}}}
\newcommand{\cyclic}{\Lambda}
\newcommand{\simpl}{\D}
\newcommand{\intvl}{\nabla}

\newcommand{\p}{p}
\newcommand{\triv}{triv}
\renewcommand{\Eq}{\text{Eq}}
\DeclareMathAlphabet{\mathpzc}{OT1}{pzc}{m}{it}

\title{Berger-Joyal duality and traces I}
\begin{document}
%\author{Nicholas Cecil and Benjamin Cooper}
\author[1]{Nicholas Cecil\thanks{ceciln\char 64 scf.edu}}
\author[2]{Benjamin Cooper\thanks{ben-cooper\char 64 uiowa.edu}}
\affil[1,2]{University of Iowa, Department of Mathematics}%, 14 MacLean Hall, Iowa City, IA 52242-1419 USA}
%\email{nicholas-cecil\char 64 uiowa.edu}
%\address{University of Iowa, Department of Mathematics, 14 MacLean Hall, Iowa City, IA 52242-1419 USA}
%\email{ben-cooper\char 64 uiowa.edu}

 \maketitle

\begin{abstract}
We give a new proof that the opposite of Joyal's disk category $\cal{D}_n$
is Berger's wreath product category $\Theta_n = \simpl\wr\cdots\wr\simpl$.
Our techniques continue to apply when the simplex category $\simpl$ is
replaced by Connes' cyclic category $\cyclic$ and some other crossed simplicial groups.
\end{abstract}

\section{Introduction}\label{sec:intro}

Joyal introduced finite combinatorial $n$-disks in his study of higher
categories \cite{JoyalCombDisks}.  These $n$-disks fit together to determine
a category $\cal{D}_n$ which he used to introduce a notion of weak
$n$-categories.  In this setting, the image of the Yoneda embedding
$\cal{D}_n^\op \hookrightarrow \fun(\cal{D}_n, \cat{Set})$ is a collection of $n$-categories. 
Berger found an inductive definition $\Theta_n$ of this category and
produced equivalences
\begin{equation}\label{eqn:bergerjoyal}
\Theta_n\simeq\cal{D}_n^\op,
\end{equation}
see \cite{BergerLoopSpace}.  We call Eqn. \eqref{eqn:bergerjoyal} Berger-Joyal duality.

The inductive definition of $\Theta_n$ is written in terms of a wreath product operation
\begin{align*}
    (-)\wr (-):\cat{Cat}_{/\G}\times \cat{Cat}\to \cat{Cat}
\end{align*}
where $\cat{Cat}_{/\G}$ is the category of pairs $C_\gamma := (C, \gamma)$ where $C$ is a small category and $\gamma : C\to \G$ is a functor from $C$ to Segal's category $\G$ (see Def. \ref{def:segal}).
In terms of this operation, Berger defined $\Theta_n$ recursively according to the formulas
\begin{align*}
  \Theta_1 &:= \simpl\\
  \Theta_n &:= \simpl\wr \Theta_{n-1}  \quad\quad\textnormal{ for }\quad n>1
  \end{align*}
where $\simpl$ is the simplex category.

Since the wreath product is functorial, it is straightforward to study the outputs of
$A\wr X$ as $A$ and $X$ vary.  The same is not true for $n$-disks.  This
makes direct generalizations of Berger-Joyal duality to other settings a
tricky business.  For example, in order to replace the simplex category
$\simpl$ by Connes' cyclic category $\cyclic$ one needs to  intuit the
cyclic analogue of an $n$-disk; while possible, this kind of ingenuity does
not represent an extensible solution.

The purpose of this paper is to give a functorial construction of Berger-Joyal duality, see Theorem \ref{thm:bjduality}.
In order to do this, we introduce a generalized wreath product
operation inspired by Kelly's theory of clubs, see Insp. \ref{insp:kelly}.  If $C$ is a 2-category which has pullbacks and a terminal object $*$ and $T : C \to C$ is a functor then there is a $T$-wreath product 2-functor
\begin{align*}
    \otimes^T : C_{/T(*)}\times C\to C
\end{align*}
see Def. \ref{def:genwrprod}. If $C$ is the category $\cat{Cat}$ of small categories then
we introduce the pair of functors
$$\M : \cat{Cat}\to \cat{Cat} \quad\quad\textnormal{ and }\quad\quad \Md :\cat{Cat}\to \cat{Cat}$$ 
in Def. \ref{def:pi} and Def. \ref{def:sigma}. By construction, the $\M$-wreath product $\otimes^{\M}$ agrees with Berger's wreath product: 
$$(-)\otimes^{\M}(-) \cong (-)\wr (-).$$
The functor $\Md$ is defined to be the functor $\M$ conjugated by the involution $-^{\op}$ of $\cat{Cat}$ and Thm. \ref{thm:diskaswreath} shows that the $\Md$-wreath product satisfies
$$\intvl\otimes^\Md\cal{D}_n\simeq \cal{D}_{n+1} \quad\quad\textnormal{ for }\quad n\geq 1$$
where $\intvl\cong\simpl^\op$ is the interval category. Prop. \ref{prop:Duality} shows that the dual pair, $\M$ and $\Md$, of functors determines an isomorphism between the associated dual pair of wreath products:
$$\left((-)\wr(-)\right)^\op \cong (-)^\op \otimes^\Md (-)^\op.$$
These ideas combine in Thm. \ref{thm:bjduality} to give a functorial proof of Berger-Joyal duality. The base case is $\cal{D}_1\simeq \simpl^{\op}$.  Assuming $\Theta_n^\op\simeq\cal{D}_n$, the induction step is the chain of equivalences
\begin{align*}
    \Theta_{n+1}^\op &=(\simpl \wr \Theta_n)^\op \\
&=(\simpl \otimes^\M \Theta_n)^\op\\
    &\cong \simpl^\op \otimes^\Md \Theta_n^\op \\
    &\simeq \intvl \otimes^\Md \cal{D}_n \\%\tag{$\star$}\\
    &\simeq \cal{D}_{n+1}.
\end{align*}

The remainder of the paper contains a few steps towards applying this
duality theorem to the study of $n$-categories.  Our motivation is to use
this construction to produce operations on higher categories.  Roughly
speaking, $n$-categories are presheaves on $\Theta_n$ with some additional
structure and, as mentioned above, our version of Berger-Joyal duality
allows us to replace a copy of $\Delta$ in the definition of $\Theta_n$ with
a category $C$ equipped with a functor $C\to \Gamma$. When there is a good
choice of functor $j : \Delta \to C$ then this placement admits induction
and restriction functors
$$PSh(\Delta\wr \cdots \wr \Delta \wr \cdots \wr \Delta) \leftrightarrows PSh(\Delta\wr \cdots \wr C \wr \cdots \wr \Delta).$$
In good cases, these functors can be used to produce operations on higher categories. 

A wreath product involving a non-standard $C$ depends on a choice of functor
$C\to \Gamma$ and \S\ref{sec:chardefsegal} contains a study of these
functors. In Berger's work, the functor $\Delta \to \Gamma$ keeps track of
the edges in the ordered sets $[n]=\{0<1<2<\cdots<n\}$, see
Def. \ref{def:segalmap}, but, as we will see, there are many other
examples. Theorem \ref{thm:sievesondelta} classifies sieves on $\Delta$, each of which determines a functor $\Delta \to \Gamma$ and Cor. \ref{cor:berger} characterizes Berger's functor as the one associated to the largest proper sieve.
 In Section \ref{sec:segalcss}, we consider the
special case in which $C:=\Delta G$ is a crossed simplicial group in the
sense of Fiedorowicz and Loday \cite{FL} and Krasauskas \cite{Kras}, and in particular  Connes' cyclic category $C:=\Lambda$, in preparation towards the sequel \cite{CecilCooper2}, in which a family of trace operations on higher categories is introduced.

 \bigskip

\noindent

\textbf{Notation.}  We use the symbols $=$ for equality, $\cong$ for isomorphism in an ambient category and $\simeq$ for equivalence in an ambient 2-category.     If $\gamma \in \cal{C}_{/X}$ is an object of the slice category then the functor $\gamma : A\to X$ will be denoted by $A_\gamma$.

\textbf{Acknowledgments.} Nicholas Cecil was supported by the Erwin and Peggy Kleinfeld Fellowship. 
The authors received support from the NSF RTG DMS-2038103 grant.

\section{Basic definitions}\label{sec:review}

Here we recall the simplex category $\Delta$, the interval category $\nabla$ and Segal's category of finite sets $\Gamma$.   Then in \S\ref{sec:finitendisks}, Joyal's categories $\cal{D}_n$ of finite combinatorial $n$-disks are reviewed. Connes' cyclic category $\Lambda$ is recalled later, in Ex. \ref{ex:cyclic_category}.

\begin{defn}{$(\simpl,\intvl)$}\label{def:simplexinterval}
    The {\em simplex category} $\simpl$ is the category with objects 
$$[n]:=\{0\leq \cdots \leq n\}$$
and order preserving set maps. The interval category $\intvl$ is the category consisting of the objects $[n]$ with $n>0$ and maps which preserve both the order and the extrema.
\end{defn}

  In Joyal's preprint, a duality between the interval category and the simplex category appears. It is the $n=1$ case of the  duality between disks and $\Theta_n$ which is proven in \S\ref{sec:bjgenwreath}. The object $[n]\in \Delta$ corresponds to the interval $\Hom_{\Delta}([n],[1])\cong [n+1]$.

\begin{prop}{(\cite[\S 1.1]{JoyalCombDisks})}\label{prop:n1duality}
$\simpl^\op\cong \intvl$
\end{prop}

Segal introduced the category $\G$ in \cite[Def. 1.1]{Segal} as a
tool for identifying infinite loop spaces. Here we recall this category
and a few properties.

\begin{defn}{($\G$)}\label{def:segal}
Segal's category $\G$ is the opposite of (the skeleton of) the category of
finite pointed sets $\FinSet_*^{\op}$.

An explicit description of the skeleton of $\G$ has objects of the form $\bn := \{ 1,\ldots, n\}$.
A map $f:\bn\to \bl$ is a set map $f:\bn\to \cal{P}(\bl)$ from $\bn$ to the power set of $\bl$ such that distinct elements $a\neq b\in \bn$ are carried to disjoint subsets $f(a)\cap f(b) = \emptyset$ of $\bl$. The composition of two maps $f : \bn\to \bl$ and $g:\bl\to \bk$ is $(gf)(a):=\cup_{b\in f(a)}g(b)$. The identity map $1_\bn : \bn \to \bn$ is $a\mapsto \{a\}$ for $a\in \bn$. 
The equivalence $P : \Gamma \xto{\sim} \FinSet_*^{\op}$ between the two descriptions is given on sets $A\in \ob(\Gamma)$ by setting $P(A):= A\sqcup \{*\}$ and, for a map $f: A\to B$ in $\Gamma$, by the pointed map $P(f):P(B)\to P(A)$ 
    \begin{align*}
        P(f)(t) :=
        \begin{cases}
            s & t\in f(s) \\
            * & t\not\in f(s) \textnormal{ for all } s
        \end{cases}
    \end{align*}

\end{defn}

\subsection{Finite combinatorial $n$-disks}\label{sec:finitendisks}
In this section we introduce Joyal's category of combinatorial disks.

\begin{defn}{($\cal{D}_n$)}\label{def:fcndisks}
An object $X$ of the category $\cal{D}_n$ of {\em finite combinatorial
  $n$-disks} consists of a collection of finite sets $X_k$ together with set
maps $s_k, t_k :X_k\to X_{k+1}$ for $0\leq k< n$ and $p_k: X_k\to X_{k-1}$
for $0< k \leq n$,
\[\begin{tikzcd}
	{X_0} & {X_1} & {X_2} & \cdots & {X_n}
	\arrow["{{s_0}}", shift left=3, from=1-1, to=1-2]
	\arrow["{{t_0}}"', shift right=3, from=1-1, to=1-2]
	\arrow["{{p_1}}"{description}, from=1-2, to=1-1]
	\arrow["{{s_1}}", shift left=3, from=1-2, to=1-3]
	\arrow["{{t_1}}"', shift right=3, from=1-2, to=1-3]
	\arrow["{{p_2}}"{description}, from=1-3, to=1-2]
	\arrow["{{s_2}}", shift left=3, from=1-3, to=1-4]
	\arrow["{{t_2}}"', shift right=3, from=1-3, to=1-4]
	\arrow["{{p_3}}"{description}, from=1-4, to=1-3]
	\arrow["{{s_{n-1}}}", shift left=3, from=1-4, to=1-5]
	\arrow["{{t_{n-1}}}"', shift right=3, from=1-4, to=1-5]
	\arrow["{{p_n}}"{description}, from=1-5, to=1-4]
\end{tikzcd}\]
    which satisfy the relations
    \begin{align}\label{eq:globulareqn}
        \p_ks_{k-1}=1_{k-1}= \p_kt_{k-1}\;\; \text{ for } 1\leq k\leq n\;\;\\ \text{ and } s_{k}s_{k-1}=t_{k}s_{k-1},\;\; s_{k}t_{k-1}=t_{k}t_{k-1} \text{ for } 1\leq k< n.
    \end{align}
In addition, we require 
    \begin{enumerate}%[(a)]
        \item $X_0 =\{*\}$ and $s_0(*)\neq t_0(*)$
\item the equalizer $\Eq(s_k, t_k) = s_{k-1}(X_{k-1})\cup t_{k-1}(X_{k-1})$ for $1\leq k < n$
        \item for each  $0\leq k < n$ and each $x\in X_{k}$, the fibre $\p_{k+1}^{-1}(x)$ is a finite linearly ordered set with minimum $s_{k}(x)$ and maximum $t_{k}(x)$.
    \end{enumerate}
A map $f:X\to Y$ in $\cal{D}_n$ is a collection of maps $f=\{f_k : X_k \to Y_k\}_{k=0}^n$ which commute with the maps $s_k$, $t_k$ and $p_k$ and preserve the linear orders.
\end{defn}

\begin{motivation}\label{ex:topdisk}
Just as the collection of topological $n$-simplices determines a functor $\simpl \to \cat{Top}$, the collection of topological  $n$-disks 
$$\bb{D}_n := \{x \in \bb{R}^n : |x| \leq 1\}$$
determine an object
\[\begin{tikzcd}
	{\mathbb{D}_0} & {\mathbb{D}_1} & {\mathbb{D}_2} & \cdots & {\mathbb{D}_n}
	\arrow["{{s_0}}", shift left=3, from=1-1, to=1-2]
	\arrow["{{t_0}}"', shift right=3, from=1-1, to=1-2]
	\arrow["{{p_1}}"{description}, from=1-2, to=1-1]
	\arrow["{{s_1}}", shift left=3, from=1-2, to=1-3]
	\arrow["{{t_1}}"', shift right=3, from=1-2, to=1-3]
	\arrow["{{p_2}}"{description}, from=1-3, to=1-2]
	\arrow["{{s_2}}", shift left=3, from=1-3, to=1-4]
	\arrow["{{t_2}}"', shift right=3, from=1-3, to=1-4]
	\arrow["{{p_3}}"{description}, from=1-4, to=1-3]
	\arrow["{{s_{n-1}}}", shift left=3, from=1-4, to=1-5]
	\arrow["{{t_{n-1}}}"', shift right=3, from=1-4, to=1-5]
	\arrow["{{p_n}}"{description}, from=1-5, to=1-4]
\end{tikzcd}\]
There are projections and inclusions 
$$\p_k:\bb{D}_k\to \bb{D}_{k-1} \quad\quad\text{ and }\quad\quad s_k,t_k:\bb{D}_{k}\to \bb{D}_{k+1}.$$
The projection $\p_{k+1}(\hat{x},x_{k+1}):=\hat{x}$ maps a vector onto its first $k$ coordinates.
The inclusions  are given by $s_k(\hat{x}):=(\hat{x},-\sqrt{1-|\hat{x}|^2})$ and  $t_k(\hat{x}) := (\hat{x}, \sqrt{1-|\hat{x}|^2})$.
The relations in Eqn. \eqref{eq:globulareqn} above are satisfied by these maps.
\end{motivation}

\begin{rem}
Notice that for an interior point $x\in\bb{D}_{k}$, 
the fibre $\p_{k+1}^{-1}(x)\cong [s_k(x), t_k(x)]$ is a non-degenerate interval with
endpoints $s_k(x)$ and $t_k(x)$.  On the other hand, for a boundary point
$x\in \del \bb{D}_{k}$, the fibre $\p_{k+1}^{-1}(x)$ is a single point.  In the
definition of a finite combinatorial disk, the topological condition that $y\in \bb{D}_{k}$ is a
boundary point is replaced by the condition that $y$ lies in the image of
$s_{k-1}$ or $t_{k-1}$, see condition (2) in Def. \ref{def:fcndisks}.
\end{rem}

\section{Wreath products}\label{sec:genwreath}

\subsection{Abstract wreath products}\label{sec:catwr}

We introduce a generalization of Berger's wreath product which depends on an endofunctor $T : \cal{C}\to\cal{C}$ of a category $\cal{C}$ with pullbacks and a terminal object $\ast$.

\begin{defn}{($\otimes^T$)}\label{def:genwrprod}
    For a functor $T:\cal{C}\to \cal{C}$, the {\em $T$-wreath product}
$$\otimes^T :	\cal{C}_{/T(*)}\times\cal{C} \to  \cal{C}$$
is the pullback $A_r\otimes^T X := A \times_{T(*)} T(X)$ in the diagram below.
\begin{center}
  \begin{tikzpicture}[node distance=2.5cm]
    \node (A) {$A_r\otimes ^T X$};
        \node (B) [right of=A] {$T(X)$};
    \node (C) [below of=A] {$A$};
    \node (D) [right of=C] {$T(*)$};
\draw[->] (A) to node[midway,above] {$\pi$} (B);
\draw[->] (A) to node [swap,midway,left] {$p$} (C);
\draw[->] (B) to node [right] {$T(!_X)$} (D);
\draw[->] (C) to node[midway, above] {$r$} (D);
    \end{tikzpicture}
  \end{center}

\end{defn}

Here the map $!_X : X\to *$ is the unique map from $X$ to the terminal object $*$ and $r : A\to T(*)$ is the structure map of the object $A_r\in \mathcal{C}_{/T(*)}$. The next remark records how $\otimes^T$ depends on the choices involved.

\begin{rem}\label{rem:Stability}%\label{Stability_Equiv}
    \begin{enumerate}[(i)]
        \item Choosing a different model for the pullback perturbs $\otimes^T$ up to natural isomorphism. 
        \item Replacing $T$ by a functor which is naturally isomorphic to $T$ perturbs the wreath product $\otimes^T$ by natural isomorphism.

\item If $\phi:A'\to A$ is an equivalence  then the induced map   
$\phi \otimes 1_X : A'_{r\phi}\otimes^T X\to A_{r}\otimes^T X$
is not necessarily an equivalence. However, in our main application, $\cal{C}=\cat{Cat}$ and $T=\M$ or $T=\Md$ below, the maps $T(!_X) : T(X)\to T(*)$ are isofibrations (Prop. \ref{prop:Mprops} (2) and Prop. \ref{def:sigma} (3)), which implies that the map $\phi \otimes 1_X$ is an equivalence.
\item In our applications, $T : \cat{Cat}\to \cat{Cat}$ is a $2$-functor. This implies that $T$ preserves equivalences of categories. Together with (iii), if $\psi : X' \to X$ is an equivalence then $1_A\otimes \psi$ is an equivalence because it is the pullback of the equivalence $T(\psi)$ between isofibrations $T(!_{X'})$ and $T(!_X)$ over $T(*)$.
\end{enumerate}

\end{rem}

The proposition below says that an automorphism $\tau$ of $\cal{C}$ carries the $T$-wreath product to the $T^\tau$-wreath product, where $T^\tau$ is the conjugate of $T$ by $\tau$. It
is the reason for studying the dual pair of functors $\M$ and $\Md$ introduced below.

\begin{prop}\label{prop:Duality}
    Let $\cal{C}$ be a 2-category and $T :\cal{C}\to\cal{C}$ is a functor  which defines a wreath product $\otimes^T$.
If $\tau : \cal{C}\to\cal{C}$ an automorphism of $\cal{C}$ and $T^{\tau} : \cal{C}\to \cal{C}$ is the functor defined by $T^\tau =\tau T\tau^{-1}$ then there is a natural isomorphism 
$$\tau\big(A_r\otimes^T B)\cong \tau(A_r)\otimes^{T^{\tau}}\tau(B).$$
\end{prop}

\begin{proof}
  Since $\tau$ is an automorphism it preserves pullbacks and the terminal object.
Applying $\tau$ to the pullback diagram in Def. \ref{def:genwrprod} gives the left-hand side below and then applying the relation $\tau T = T^{\tau} \tau$ with $\tau(!_X)=!_{\tau(X)}$ gives the right-hand side. The result follows from Rmk. \ref{rem:Stability} (i) above.
\begin{center}
  \begin{tikzpicture}[node distance=2.5cm]
    \node (A) {$\tau(A_r\otimes ^T X)$};
        \node (B) [right of=A] {$\tau T(X)$};
    \node (C) [below of=A] {$\tau(A)$};
    \node (D) [right of=C] {$\tau T(*)$};

    \node (ZZZ) [right of=B] {};

    \node (A1) [right of=ZZZ] {$\tau(A_r\otimes ^T X)$};
        \node (B1) [right of=A1] {$T^\tau(\tau(X))$};
    \node (C1) [below of=A1] {$\tau(A)$};
    \node (D1) [right of=C1] {$T^\tau(*)$};

\draw[->] (A) to node [above] {$\tau(\pi)$} (B);
\draw[->] (A) to node [swap,left] {$\tau(p)$} (C);
\draw[->] (B) to node [right]{$\tau T(!_X)$} (D);
\draw[->] (C) to node [above] {$\tau(r)$} (D);

\draw[->] (A1) to node [above] {$\tau(\pi)$} (B1);
\draw[->] (A1) to node [swap,left] {$\tau(p)$} (C1);
\draw[->] (B1) to node [right]{$T^\tau(!_{\tau(X)})$} (D1);
\draw[->] (C1) to node [above]{$\tau(r)$} (D1);

    \end{tikzpicture}
  \end{center}

\end{proof}

\begin{insp}{(Clubs)}\label{insp:kelly}
For a category $C$, the category of endofunctors $[C, C]$ is
monoidal under composition. If $T$ is a monad then the overcategory
$[C,C]_{/T}$ is also monoidal. When $C$ has pullbacks and $T$ satisfies some conditions, Kelly showed \cite{Kelly_Clubs_V} that this monoidal structure restricts to a full subcategory of $[C,C]_{/T}$ which is equivalent to $C_{/T(*)}$ via the evaluation at $*$. Our $T$-wreath product in Def. \ref{def:genwrprod} is modelled on the resulting monoidal structure on $C_{/T(*)}$
  \end{insp}

\subsection{Berger's wreath product and $\Theta_n$}\label{sec:catwr}

 In this section, we introduce the wreath product $A\wr B: =A\otimes^{\M} B$ by setting $T:=\M$
in the construction from 
Def. \ref{def:genwrprod}.  The product $\otimes^\M$ determined by the functor $\M$ agrees with the wreath product in the literature, see \cite{BergerLoopSpace} or \cite{AyalaConfig,rezk,bergner}.

\begin{defn}\label{def:pi}
There is a functor $\M : \cat{Cat} \to \cat{Cat}$. For any category $C$, the category $\M(C)$ is given by the data below.
    \begin{enumerate}
        \item An object of $\M(C)$ is a pair $(I,a)$ where $I = \{1,2,\ldots,n\} \in \G$ is a set and $a:I\to \ob(C)$ is a function.
        \item A map $f:(I,a)\to (J,b)$ in $\M(C)$ consists of a collection $f = (f_0, \{f_{ji}\})$ where $f_0:I\to J$ in $\G$ and for each $i\in I$ and $j\in f_0(i)$, $f_{ji}:a(i)\to b(j)$ is a map in $C$.
        \item The composition $gf:(I,a)\to (K,c)$ of $f:(I,a)\to (J,b)$ and $g:(J,b)\to (K,c)$ is given by the pair $gf=((gf)_0, \{(gf)_{ki}\})$ with  $(gf)_0:=g_0f_0$ and $(gf)_{ki}:=g_{kj}f_{ji}$ where $j\in f_0(i)$ is the unique element such that $k\in g_0(j)$.
        %% \begin{align*}
        %%     (gf)_0=g_0f_0\quad\quad\text{ and }\quad\quad (gf)_{ki}=g_{kj}f_{ji}
        %% \end{align*}

      \item The identity map $1_{(I,a)} : (I,a)\to (I,a)$ is the map $1_{(I,a)} := (1_0, 1_{ji})$ where $1_0$ is the identity map $1_0 := 1_I$ on $I$ and $1_{ii} : a(i) \to a(i)$ is the identity map $1_{ii} := 1_{a(i)}$.
    \end{enumerate}

 If $F : C \to D$ is a functor then there is a functor  $\M(F) : \M(C)\to \M(D)$. On objects $\M(F)(I,a) := (I, F(a))$. If $f = (f_0, \{f_{ji}\})$ then $\M(F)(f) := (f_0, \{F(f_{ji})\})$.     
\end{defn}

The proposition below gives us some basic properties of $\M$. In particular, (2) below addresses the isofibration condition mentioned in (iii) of Rmk. \ref{rem:Stability}.

\begin{prop} \label{prop:Mprops}
  \begin{enumerate}
    \item  $\M(*)\cong \G$
\item If $X$ is a category and $!_X : X \to *$ is the canonical map then the map  $\M(!_X) : \M(X)\to \M(*)$ is an isofibration.
\end{enumerate}
\end{prop}
\begin{proof}
For (1), since $*$ consists of one object $o$ and $End_*(o)=\{1_o\}$, there is one map $\triv : I \to \{o\}$. So the objects of $\M(*)$ are pairs $(I,\triv)$ with $I\in \G$. A map $f : (I,\triv) \to (J,\triv)$ is a pair $f = (f_0, \{f_{ji}\})$ where the maps $f_{ji}: \triv(i)\to \triv(j)$ are all required to be the identity $1_o$. So $f$ is determined by $f_0$ and $(I,\triv)\mapsto I$ is an isomorphism $\M(*)\cong \Gamma$.

For (2), isomorphisms in $\FinSet_*$ are permutations so if $f : \M(!_X)(I,a)  \to J$ is an isomorphism in $\G$ then $f(s)$ is a singleton for all $s\in I$. When $j\in f(i)$ set $b(j) := a(i)$, so there is an isomorphism $\tilde{f} : (I, a) \to (J,b)$ which satisfies $\M(!_X)(\tilde{f})=f$ given by $\tilde{f} = (f, \{f_{ji}\})$ where $f_{ji} := 1_{a(i)} : a(i) \to a(i)$.
  \end{proof}

Following Def. \ref{def:genwrprod}, the functor $\M$ leads to a product $\otimes^\M$. The definition below contains the details.

\begin{defn}{($\otimes^\M$ or $\wr$)}\label{def:wrprod}
The {\em wreath product} is a 2-functor $\otimes^\M : \cat{Cat}_{/\G}\times \cat{Cat}\to \cat{Cat}$.
  If $X$ and $A$ are categories and $\gamma : X \to \G$ is a functor then $X_\gamma \otimes^\M A := X_\gamma \times_{\Gamma} \M(A)$ is the pullback of the diagram below.
\begin{center}
  \begin{tikzpicture}[node distance=2.5cm]
    \node (A) {$X\otimes^\M A$};
        \node (B) [right of=A] {$\M(A)$};
    \node (C) [below of=A] {$X$};
    \node (D) [right of=C] {$\G$};
\draw[->] (A) to node {} (B);
\draw[->] (A) to node {} (C);
\draw[->] (B) to node {} (D);
\draw[->] (C) to node[above] {$\gamma$} (D);
    \end{tikzpicture}
  \end{center}

    More concretely, an object of $X\otimes^\M A$ is a pair $(I,a)$ where $I$ is an object of $X$ and $a : \gamma(I)\to \ob(A)$ is a function.  A morphism $f:(I,a)\to (J,b)$ consists of $f_0:I\to J$ in $X$ and for each $i\in \gamma(I)$ and $j\in \gamma(f_0)(i)$ a map $f_{ji} : a(i)\to b(j)$ in $A$. This construction agrees with \cite[Def. 3.1]{BergerLoopSpace}, so we also write $X\wr A = X\otimes^\M A$.

\end{defn}

\begin{defn}{(Segal map)}
If $X_\gamma \in \cat{Cat}_{/\G}$ is an object then the functor $\gamma:X\to \G$ is called the {\em Segal map}.
\end{defn}

\subsection{Wreath products of the form $\simpl\wr C$}

 We now give an account of wreath products of the form $\simpl\wr C$. 
The Segal map $\gamma$ given by the definition below records the edges of $[n]$.

\begin{defn}{($\gamma : \Delta \to \G$)}\label{def:segalmap}
For  $[n]\in \simpl$, let $E([n]) := \{e_1,e_2,\ldots,e_{n}\}$ be the set of edges $e_i : i-1 \to i$
  which generate the poset $[n] = \{0<1<2<\cdots<n\}$ as a category.  On objects the Segal map is defined by setting $\gamma([n]) := E([n])$.

Now notice that, for each $\phi\in \Hom_{[n]}(i,j)$, there is a unique subset
$E(\phi)\subseteq E([n])$ of edges whose composition is $\phi$.
If $f : [n]\to [m]$ is a map in $\simpl$ and $e\in E([n])$ then the Segal
map is given by $\gamma(f)(e) := E(f(e)) \subseteq \gamma([m])$.
\end{defn}

Following Def. \ref{def:wrprod}, the wreath product $\simpl\wr C$ is a category with objects $([n],c)$ where the map 
 $c : E([n])\to \ob(C)$ labels each edge of $[n]$ by an object $c_i := c(i-1\to i)$ of $C$.
\[\begin{tikzcd}
	0 & 1 & \cdots & n
	\arrow["{c_1}", from=1-1, to=1-2]
	\arrow["{c_2}", from=1-2, to=1-3]
	\arrow["{c_n}", from=1-3, to=1-4]
\end{tikzcd}\]

A map $f : ([n],c) \to ([m],d)$ consists of a map $f_0 : [n]\to [m]$ in $\simpl$ together with maps $\{f_{ji} : c(i)\to d(j)\}_{j\in f_0(i)}$ in $C$. The diagram below is a picture of a map  $f:([3], c)\to ([4], d)$

\[\begin{tikzcd}
	0 & 1 & 2 & 3 \\
	0 & 1 & 2 & 3 & 4
	\arrow[""{name=0, anchor=center, inner sep=0}, "{c_1}", from=1-1, to=1-2]
	\arrow[from=1-1, to=2-1]
	\arrow["{c_2}", from=1-2, to=1-3]
	\arrow[from=1-2, to=2-2]
	\arrow[""{name=1, anchor=center, inner sep=0}, "{c_3}", from=1-3, to=1-4]
	\arrow[from=1-3, to=2-2]
	\arrow[from=1-4, to=2-4]
	\arrow[""{name=2, anchor=center, inner sep=0}, "{d_1}"', from=2-1, to=2-2]
	\arrow[""{name=3, anchor=center, inner sep=0}, "{d_2}"', from=2-2, to=2-3]
	\arrow[""{name=4, anchor=center, inner sep=0}, "{d_3}"', from=2-3, to=2-4]
	\arrow["{d_4}"', from=2-4, to=2-5]
	\arrow["{f_{1,1}}"{description}, shorten <=4pt, shorten >=4pt, dashed, from=0, to=2]
	\arrow["{f_{2,3}}"{description}, shorten <=7pt, shorten >=7pt, dashed, from=1, to=3]
	\arrow["{f_{3,3}}"{description}, shorten <=4pt, shorten >=4pt, dashed, from=1, to=4]
\end{tikzcd}\]
where the vertical solid arrows depicts the morphism $f_0$ in $\simpl$ and the $f_{j,i}$ are morphisms in $C$. When maps are illustrated in this way, their composition is given by stacking the diagrams and composing those interfacing constituents.

Lastly, we use the wreath product to define $\Theta_n$ below.

\begin{defn}{($\Theta_n$)}\label{def:thetan} \label{eq:thetaadd}
$\Theta_1 := \simpl$ and  $\Theta_n := \simpl\wr \Theta_{n-1}$ for $n>1$.
\end{defn}

\subsection{Combinatorial disks as wreath products}\label{sec:combdisk}

Following Prop. \ref{prop:Duality}, there is a conjugate functor 
\begin{equation}\label{eq:dualfun} \Md(C) := \M(C^{\op})^{\op} \end{equation}
of $\M$ with respect to the automorphism  $-^{\op} : \cat{Cat}\to \cat{Cat}$.
Thm. \ref{thm:diskaswreath} below is the opposite $\intvl\otimes^\Md \cal{D}_n\simeq\cal{D}_{n+1}$ of Berger's $\Theta_n$ recursion in Def. \ref{def:thetan} above. Before proving this theorem, the functor $\Md$ is discussed in more detail.

    If $(A, \ast)$ is a pointed set then let $A\backslash\ast \subset A$ be the result of removing the basepoint from the set $A$. In the proposition below, objects of $\Gamma^{\op}$ are identified with finite sets using the equivalence $P$ of Def. \ref{def:segal}.

\begin{prop}\label{def:sigma}
\begin{enumerate}
\item For a category $C$, the category $\Md(C)$ from Eqn. \eqref{eq:dualfun} has objects the pairs $(I,a)$ with $I\in \G^{\op}\simeq \FinSet_*$ and $a:I\backslash\ast \to \ob(C)$.
    \begin{itemize}
        \item A map $f:(I,a)\to (J,b)$ is a pair $f=(f_0, \{f_{ji}\})$ where $f_0:I\to J$ is a pointed map, for each $i\in I\backslash \ast$ such that $j=f_0(i)\in J\backslash\ast$, $f_{ji}:a(i)\to b(j)$ is a map in $C$.
        \item If $f:(I,a)\to (J,b)$ and $g:(J,b)\to (K,c)$ then $gf$ is given by $(gf)_0=g_0f_0$ and $(gf)_{ki}=g_{kj}f_{ji}$ where $k=g_0(j)$ and $j=f_0(i)$.
    \end{itemize}
\item $\Md(*)\cong \Gd$
\item If $X$ is a category and $!_X : X \to *$ is the canonical map then the map  $\Md(!_X) : \Md(X)\to \Md(*)$ is an isofibration.
\end{enumerate}
\end{prop}
\begin{proof}
  Statement (1) comes from unwinding the definitions, while statements (2) and (3) follow from Prop. \ref{prop:Mprops} and the involution $-^{\op} : \cat{Cat}\to \cat{Cat}$. In more detail,
  \begin{enumerate}
\setcounter{enumi}{1}
\item $\Md(*) = (\M(*^{\op}))^\op = (\M(*))^\op = \G^\op$
\item By Prop. \ref{prop:Mprops}, $\M(!_{X^\op}) : \M(X^{\op}) \to \G$ is an isofibration and $-^{\op}$ 
carries isofibrations to isofibrations,
 $\Md(!_X)=\M(!_{X^\op})^{\op} : \Md(X) = (\M(X^{\op}))^\op \to \G^\op = \Md(*)$ is an isofibration.
    \end{enumerate}

  \end{proof}

\begin{rem}
    Berger showed that $\Pi(C)$ is the result of freely closing $C$ under finite products and adding a zero object \cite[Lem 3.2]{BergerLoopSpace}.  So the category $\Pi^\op(C)$ can be thought of as freely closing $C$ under finite coproducts and adding a zero object.  If $a:I\backslash\ast \to \ob(C)$ is a family of objects of $C$ then $(I,a)\in \Pi^\op(C)$ is their coproduct.
\end{rem}

If $X_\omega \in \cat{Cat}_{/\G^{\op}}$ is an object then $\omega:X\to \G^\op$ is called the {\em coSegal map}. We will write $X\otimes^{\Md} A$ instead of $X_\omega \otimes^{\Md} A$ when it is unambiguous. Unwinding Def. \ref{def:genwrprod} with $T:=\Md$ produces the definition below.

\begin{defn}
The dual wreath product or cowreath product $X_\omega \otimes^\Md A := X_\omega \times_{\Gd} \Md(A)$ is the category with objects given by pairs $(x,a)$ where $x\in X$ and $a:\omega(x)\backslash\ast \to \ob(A)$ is a function.
\begin{itemize}
        \item A map $f:(x,a)\to (y,b)$ is a pair $(f_0, \{f_{ji}\})$ where $f_0 : x\to y$ is a map in $X$ and, for each $i\in \omega(x)\backslash\ast$ such that $\omega(f_0)(i)=j\in \omega(y)\backslash\ast$, there is a map $f_{ji} : a(i)\to b(j)$ in $A$. 
    \end{itemize}
\end{defn}

\begin{rem}
We have found two instances of the cowreath product in the literature.  For a category $C$, Borceux introduced a construction $\cat{Set}(C)$; its finite analogue $\cat{FinSet}(C)$ is the cowreath product $\cat{FinSet}\otimes^{\Pi^\op}C$ where the coSegal functor $-_+ : \cat{FinSet}\to \cat{FinSet}_*$ freely adjoins a basepoint, see \cite[Ch. 8]{Borceux94}. 
Lurie uses a category $\Delta_S = \Delta_\omega\otimes^{\Pi^\op}S$ where $S$ is a set and the coSegal functor $\omega : \Delta\to \cat{FinSet}$ is given by $\omega([n]):=\{0,\ldots,n\}$, see  \cite[Def. 2.1.1]{Lurie09_Goodwillie}.
\end{rem}

\begin{prop}\label{def:cosegalmap}
If $\g : \simpl \to \G$ is the Segal functor from Def. \ref{def:segalmap}
then, under the identifications in $\Delta^{\op}\cong \intvl$ of  Prop. \ref{prop:n1duality} and $\Gamma^{\op}\simeq \FinSet_*$ of Def. \ref{def:segal}, the coSegal functor $\g^{\op} : \simpl^{\op} \to \G^{\op}$ admits a description as a functor $\omega : \intvl\to \FinSet_*$ which records the non-extreme points of an interval. More precisely,
    \begin{itemize}
       \item  If $[n] \in \intvl$ then $\omega([n]):=(\{1,2,\ldots,n-1, *\},\{*\})$ is the set of non-extreme points.
       \item  If $f:[n]\to [m]$ is a map of intervals then $\omega(f):\omega([n])\to \omega([m])$ in $\FinSet_*$
        \begin{align*}
            \omega(f)(i) :=
            \begin{cases}
                f(i) & \textnormal{ if } f(i) \not\in \{0,m\} \\
                * & \textnormal{ if } f(i) \in \{0,m\} \\
            \end{cases}
        \end{align*}
    \end{itemize}
\end{prop}

The coSegal map $\omega : \intvl \to \Gd$ allows us to introduce a cowreath product 
 $\intvl\otimes^\Md X$ for any category $X$. The theorem below shows that Joyal's disk categories from Def. \ref{def:fcndisks} admit an inductive definition in terms of this product.

\begin{thm}\label{thm:diskaswreath}
  The disk categories are iterated cowreath products: there are equivalences
    \begin{align}
        \intvl \simeq \cal{D}_1  \label{eq:1daw}\\
        \intvl_\omega \otimes^\Md \cal{D}_n\simeq \cal{D}_{n+1}.  \label{eq:2daw}
    \end{align}
\end{thm}

\begin{proof}
  For the base case Eqn. \eqref{eq:1daw}, by Def. \ref{def:fcndisks} an object $X \in \cal{D}_1$ is a diagram 
\[\begin{tikzcd}
	{\{*\}} & {X_1}
	\arrow["{s_0}", shift left=3, from=1-1, to=1-2]
	\arrow["{t_0}"', shift right=3, from=1-1, to=1-2]
	\arrow["{p_1}"{description}, from=1-2, to=1-1]
\end{tikzcd}\]
such that $\p_1^{-1}(*) = X_1 = \{ s_0(*) \leq 1 \leq 2 \leq \cdots \leq t_0(*) \}$ is a finite linearly ordered set with minimum $s_0(*)$ and maximum $t_0(*)$. A map $f : X\to Y$ in $\cal{D}_1$ is determined by the order and endpoint preserving map $f_1 : X_1\to Y_1$.
As $\intvl$ is a skeleton of the category of finite linearly ordered sets containing at least two elements and order and extrema preserving maps, the functor $f : \cal{D}_1 \to \intvl$ determined by the assignment $f(X) := [\#X_1]$ is an equivalence of categories.

Now for the induction Eqn. \eqref{eq:2daw}, there is a functor
\begin{equation}\label{eq:inductioniso}
\Phi:\cal{D}_{n+1}\to \intvl_\omega \otimes^\Md \cal{D}_n \quad\textnormal{ given by }\quad \Phi(X) := (X_1, a)
\end{equation}
where $a(i) := \t^iX \in \ob(\cal{D}_n)$ assigns to each $i\in \omega(X_1)$ ($i\in X_1$ and $i$ non-extreme) an $n$-disk $\t^iX$. 

We now check that the definition of $\Phi$ makes sense; the idea is that an $(n+1)$-disk is an interval $X_1$ together with an $n$-disk hanging from each of its interior points. Recall that an $(n+1)$-disk $X\in \cal{D}_{n+1}$ is determined by a collection of data
\[\begin{tikzcd}
	{X_0} & {X_1} & {X_2} & \cdots & {X_n} & {X_{n+1}}
	\arrow["{{{s_0}}}", shift left=3, from=1-1, to=1-2]
	\arrow["{{{t_0}}}"', shift right=3, from=1-1, to=1-2]
	\arrow["{{{p_1}}}"{description}, from=1-2, to=1-1]
	\arrow["{{{s_1}}}", shift left=3, from=1-2, to=1-3]
	\arrow["{{{t_1}}}"', shift right=3, from=1-2, to=1-3]
	\arrow["{{{p_2}}}"{description}, from=1-3, to=1-2]
	\arrow["{{{s_2}}}", shift left=3, from=1-3, to=1-4]
	\arrow["{{{t_2}}}"', shift right=3, from=1-3, to=1-4]
	\arrow["{{{p_3}}}"{description}, from=1-4, to=1-3]
	\arrow["{{{s_{n-1}}}}", shift left=3, from=1-4, to=1-5]
	\arrow["{{{t_{n-1}}}}"', shift right=3, from=1-4, to=1-5]
	\arrow["{{{p_n}}}"{description}, from=1-5, to=1-4]
	\arrow["{s_{n-1}}", shift left=3, from=1-5, to=1-6]
	\arrow["{t_n}"', shift right=3, from=1-5, to=1-6]
	\arrow["{p_{n+1}}"{description}, from=1-6, to=1-5]
\end{tikzcd}\]
satisfying the added conditions in Def. \ref{def:fcndisks}. As before, $X_1 \in \intvl$ is an interval. For each non-extreme element $i\in X_1$, an $n$-disk $\t^i X$ can be extracted from the $(n+1)$-disk $X$ by setting
$$\t^iX_0 := \{i\} \quad\textnormal{ and }\quad \t^iX_k := \p_{k+1}^{-1}(\t^iX_{k-1}) \quad\textnormal{ for }\quad  0< k \leq n$$
in the diagram below.
\[\begin{tikzcd}
	{\tau^iX_0} & {\tau^iX_1} & {\tau^iX_2} & \cdots & {\tau^iX_n} \\
	{X_1} & {X_2} & {X_3} && {X_{n+1}}
	\arrow["{{{s^i_0}}}", shift left=3, from=1-1, to=1-2]
	\arrow["{{{t^i_0}}}"', shift right=3, from=1-1, to=1-2]
	\arrow["\subseteq"{marking, allow upside down}, draw=none, from=1-1, to=2-1]
	\arrow["{{{p^i_1}}}"{description}, from=1-2, to=1-1]
	\arrow["{{{s^i_1}}}", shift left=3, from=1-2, to=1-3]
	\arrow["{{{t^i_1}}}"', shift right=3, from=1-2, to=1-3]
	\arrow["\subseteq"{marking, allow upside down}, draw=none, from=1-2, to=2-2]
	\arrow["{{{p^i_2}}}"{description}, from=1-3, to=1-2]
	\arrow["{{{s^i_2}}}", shift left=3, from=1-3, to=1-4]
	\arrow["{{{t^i_2}}}"', shift right=3, from=1-3, to=1-4]
	\arrow["\subseteq"{marking, allow upside down}, draw=none, from=1-3, to=2-3]
	\arrow["{{{p^i_3}}}"{description}, from=1-4, to=1-3]
	\arrow["{{{s^i_{n-1}}}}", shift left=3, from=1-4, to=1-5]
	\arrow["{{{t^i_{n-1}}}}"', shift right=3, from=1-4, to=1-5]
	\arrow["{{{p^i_n}}}"{description}, from=1-5, to=1-4]
	\arrow["\subseteq"{marking, allow upside down}, draw=none, from=1-5, to=2-5]
\end{tikzcd}\]
The structure maps $\p^i_k : \t^i X_k \to \t^i X_{k-1}$ and $s^i_k, t^i_k : \t^i X_k \to \t^i X_{k+1}$ are given by restricting those of $X$: $\p^i_k := \p_{k+1}|_{\t^i X_{k}}$, $s^i_k := s_{k+1}|_{\t^i X_{k}}$ and $t^i_k := t_{k+1}|_{\t^i X_{k}}$. The relations in Eqn. \eqref{eq:globulareqn} of Def. \ref{def:fcndisks} hold for these assignments because they hold for those of $X$. By definition $\#\t^iX_0 = \#\{i\} = 1$, conditions $(1)$, $(2)$ and $(3)$ in Def. \ref{def:fcndisks} are addressed as follows:
\begin{enumerate}
\item Since $s_1(x) = t_1(x) \Leftrightarrow x \in s_0(*) \cup t_0(*)$, $\Eq(s^i_0, t^i_0)\subset \Eq(s_1, t_1) = \emptyset$ because $i\in X_1$ is an not extreme point.
  
\item $\Eq(s^i_k, t^i_k) = \{ x \in \t^i X_k : s^i_k(x) = t^i_k(x) \} = \t^iX_k \cap \{ x\in X_{k+1} : s_{k+1}(x) = t_{k+1}(x) \} = \t^iX_k\cap (s_k(X_k)\cup t_k(X_k))= s^i_{k-1}(\t^iX_{k-1}) \cup t^i_{k-1}(\t^iX_{k-1})$

\item If $x\in \t^iX_{k-1} \subset X_k$ then $(\p^i_k)^{-1}(x) = \p_{k+1}^{-1}(x) = \{ s_{k}(x) \leq 1 \leq \cdots \leq t_{k}(x) \} = \{ s^i_{k-1}(x) \leq 1 \leq \cdots \leq t^i_{k-1}(x) \}$.
\end{enumerate}

Now if $f:X\to Y$ is a map of combinatorial $(n+1)$-disks then by Def. \ref{def:fcndisks} $f = \{ f_k : X_k \to Y_k \}_{k=0}^{n+1}$ is a collection of maps which commute with the structure maps $s_k$, $t_k$, $\p_k$ of $X$ and $f_{k+1} : \p_{k+1}^{-1}(x) \to \p_{k+1}^{-1}(f_{k}(x))$ preserves the linear order for all $0\leq k \leq n$ and for all $x \in X_{k}$. So for each $i\in \omega(X_1)$ such that $f_1(i) = j\in \omega(Y_1)$, there is a restriction 
$\t^{ji}f := \{ f_k|_{\t^i X}\}_{k=0}^n : \t^iX \to \t^j Y$.
 This restriction is functorial, if $f : X\to Y$ and $g:Y\to Z$ are maps of $(n+1)$-disks such that $f_1(i)=j\in \omega(Y_1)$ and $g_1(j)=k\in \omega(Z_1)$  then $\t^{ki}(gf)=\t^{kj}(g)\t^{ji}(f)$.

So we conclude that Eqn. \eqref{eq:inductioniso} defines a functor $\Phi:\cal{D}_{n+1}\to \intvl\otimes^\Md\cal{D}_n$
 which assigns to $(n+1)$-disks $X$, $\Phi(X) := (X_1, a)$ and maps $f : X\to Y$ between $(n+1)$-disks $\Phi(f) := (\Phi(f)_0, \{\Phi(f)_{ji}\})$ where $\Phi(f)_0 := f_1$ and $\Phi(f)_{ji} := \t^{ji}f$.

To show that $\Phi$ is an equivalence of categories, we prove (1) $\Phi$ is fully faithful and (2) $\Phi$ is essentially surjective.
\begin{enumerate}
\item $\Phi$ is fully faithful. There are mutually inverse maps
  $$ \alpha : \Hom_{\cal{D}_{n+1}}(X,Y)\leftrightarrows \Hom_{\intvl\otimes^{\Md} \cal{D}_n}((X_1,a),(Y_1,b)): \beta$$
between sets of morphisms. For $\alpha$, $\alpha(f) := \Phi(f) = (\Phi(f)_0, \Phi(f)_{ji})= (f_1, \{ \t^{ji} f\})$ as discussed above. For $\beta$, if $g : (X_1,a) \to (Y_1, b)$ then $g = (g_0, \{ g_{ji}\})$, so $\beta(g) : X\to Y$ is defined by $\beta(g) := \{\beta(g)_k :X_k\to Y_k\}_{k=0}^{n+1}$ where $\beta(g)_0(*) := *$, $\beta(g)_1 := g_0$ and, for $k\geq 2$, the map 
$$\beta(g)_k : X_k \to Y_k \quad\text{ is defined by }\quad \beta(g)_k(x) := (g_{ji})_{k-1}(x)$$ 
for $x\in \t^iX_{k-1}$, also $\beta(g)_k(s_{k-1}\cdots s_0(*)) := s_{k-1}\cdots s_0(*)$ and $\beta(g)_k(t_{k-1}\cdots t_0(*)):= t_{k-1}\cdots t_0(*)$ since $X_k = \sqcup_{i\in \omega(X_1)} \t^iX_{k-1} \sqcup \{s_{k-1}\cdots s_0(*)\} \sqcup \{t_{k-1}\cdots t_0(*)\}$.

On one hand, the composition 
$\alpha\beta(g) = \alpha(1_*, \beta(g)_0, \{\beta(g)_k\}_{k=2}^\infty) = (g_1, \{ \tau^{ji} \beta(g)_k\}) = g$ so that $\alpha\beta=1$.
On the other hand, the composition $\beta\alpha(f) = \beta\Phi(f) = \beta(f_0, \{\t^{ji} f\}) = (1_*, \beta(\Phi(f))_1, \{\beta(\Phi(f))_k\}_{k=2}^\infty)$ where $\beta(\Phi(f))_1 = f_1$ and $\beta(\Phi(f))_k(x) = (\t^{ji} f)_{k-1}(x) = f_k(x)$ when $x\in \t^iX_{k-1}$, again since $X_k = \sqcup_{i\in \omega(X_1)} \t^iX_{k-1}$, $\beta(\Phi(f))_k = f_k$ so that $\beta\alpha = 1$.

\item $\Phi$ is essentially surjective. Suppose that
  $X\in\cal{D}_{n+1}$ is an $(n+1)$-disk. Then there is an equivalence
  $\phi : X_1\xto{\sim} [\ell]$ for some $\ell \in \mathbb{Z}_{\geq 0}$.  By setting $\phi_0 := \phi$ and $\phi_{\phi(i)i} := 1_{\t^iX} : \t^i X\to \t^{\phi(i)} X$, this extends to an equivalence $\tilde{\phi} := (\phi_0,\phi_{ji}) : (X_1,a) \xto{\sim} ([\ell],\tilde{a})$ where $\tilde{a}(i) := \t^i X$ for $i\in \omega([\ell]) = \{ 1,2,\ldots, \ell-1\}$.
\end{enumerate}
\end{proof}

%Other perspectives on Berger-Joyal duality can be found in the references \cite{DisksOmegaDuality} and \cite{Oury2010}.

\section{Berger-Joyal duality}\label{sec:bjgenwreath}

With the notation for both the wreath product $X\wr A = X\otimes^{\M} A$ and the cowreath product $Y \otimes^{\Md} B$ in mind, the statement of Prop. \ref{prop:Duality} becomes the corollary below.

\begin{cor}\label{prop:wreath_duality}
    For $X_\gamma\in \cat{Cat}_{/\G}$ and any category $A$, there is a natural isomorphism
    \begin{align*}
        (X_\gamma \otimes^{\M} A)^\op \cong X^\op_{\gamma^\op} \otimes^\Md A^\op
    \end{align*}
\end{cor}

We now have everything that we need to reprove Berger-Joyal duality.

\begin{thm}\label{thm:bjduality}
    For each $n\in\mathbb{Z}_{\geq 1}$, there is equivalence $\Theta_n^\op\simeq \cal{D}_n$.
\end{thm}
\begin{proof}
    The proof is by induction. When $n=1$, there is an equivalence $\cal{D}_1\simeq \intvl$ by Thm. \ref{thm:diskaswreath} so that $\cal{D}_1\simeq \simpl^{\op}$ by Prop. \ref{prop:n1duality}. Now assuming that $\Theta_n^\op\simeq \cal{D}_n$, there is a chain of equivalences
    \begin{align*}
        \Theta_{n+1}^\op &=(\simpl\wr \Theta_n)^\op \tag{Def. \ref{def:thetan}} \\
        &\cong \simpl_{\gamma^\op}^\op \otimes^\Md \Theta_n^\op \tag{Cor. \ref{prop:wreath_duality}} \\
        &\simeq \simpl_{\gamma^\op}^\op \otimes^\Md \cal{D}_n  \tag{Induction, Rmk. \ref{rem:Stability} (iv)} \\
        &\simeq \intvl_{\omega} \otimes^\Md \cal{D}_n \tag{Prop. \ref{def:cosegalmap}, Rmk. \ref{rem:Stability}(iii)} \\
        &\simeq \cal{D}_{n+1} \tag{Thm. \ref{thm:diskaswreath}}
    \end{align*}
\end{proof}

Other perspectives on Berger-Joyal duality can be found in the references \cite{DisksOmegaDuality} and \cite{Oury2010}.

\newcommand{\Gp}{G}
\newcommand{\DG}{\Delta \Gp}
\newcommand{\gG}{\g_\Gp}
\newcommand{\Cyc}{\Lambda}

\section{Segal functors for locally finite categories}\label{sec:chardefsegal}

In order to apply the Berger-Joyal duality Theorem \ref{thm:bjduality} in
new settings, it is necessary to view categories $C$ as categories $C\to \G$
over $\G$. 
In the first part of this section, Thm. \ref{thm:sievesondelta} and Cor. \ref{cor:berger} characterize
Berger's Segal functor among those which arise as sieves. This motivates the introduction of Segal functors for other categories.
In the second part of this section we will give some examples of Segal
functors for crossed simplicial groups $\DG$.  A crossed simplicial group
$\DG$ is a category $C$ which is an extension of $\Delta$ by a collection of groups.

\subsection{A characterization of Berger's Segal functor}\label{sec:berger}

We would like to understand Segal functors $C\to \Gamma$ and our problem is
that there are too many of them.  Recall that a {\em sieve} $S$ on an object
$x\in C$ is a subfunctor $S \subseteq Y_x$ of the Yoneda functor 
\begin{equation}\label{yonedaeq}
  Y_x(y) := \Hom_C(y,x).
  \end{equation}
If $C$ is locally finite then every sieve $S$ determines a functor $S^\op : C\to \Gamma$
because $S : C^\op \to \FinSet$.  This does not use the added basepoint $*$ in
a non-trivial way (as in Berger's Def. \ref{def:segalmap}). What we really
want is Segal functors which are ``like Berger's Segal functor.'' To achieve
this Prop. \ref{prop:classicsegal} constructs Berger's Segal functor from a
sieve and Thm. \ref{thm:sievesondelta} classifies sieves on $\Delta$. Combining
these results shows that Berger's Segal functor arises from the largest proper sieve.

\begin{const}\label{construction}
If $F:C\to \FinSet$ is a functor and $S \subseteq F$ is a subfunctor then there is a quotient functor $F/S : C\to \FinSet_*$ defined as follows.

The functor $F/S:C\to \FinSet_*$ is defined on objects $x\in C$ by $(F/S)(x) := \left(F(x)\backslash S(x)\right) \sqcup\{*\}$.  If $f:x\to y$ is a map in $C$ then $(F/S)(f):(F/S)(x)\to (F/S)(y)$  is the map given by 
$$(F/S)(f)(t) := F(f)(t) \quad\textnormal{ when }\quad t\in F(x) \textnormal{ and } F(f)(t)\in (F/S)(y)$$ 
and $(F/S)(f)(t):=*$ otherwise.
\end{const}

    For a sieve $S$ on an object $x\in C$, this construction gives a functor $Y_x/S:C^\op\to \FinSet_*$ and so determines a functor $(Y_x/S)^\op : C\to \Gamma$ as discussed above.

Next Prop. \ref{prop:classicsegal} checks that Berger's Segal functor
from Def. \ref{def:segalmap} is obtained by the construction in
Const. \ref{construction}.

In order to state the proposition below we compose Berger's Segal functor $\gamma : \Delta \to\G$ with the equivalence $P : \Gamma \xto{\sim} \FinSet_*^{\op}$  from Def. \ref{def:segal} above, this gives a functor
\begin{equation}\label{eq:gammaprime}
  \gamma' := P\circ \gamma : \Delta \to \FinSet_*^{\op}.
  \end{equation}
On objects, $\gamma'([n]) = E([n])_+$ consists of the edges of $[n]$ and the basepoint $*$. For a map $f : [n]\to [m]$ in $\Delta$, the pointed map $\gamma'(f) : \gamma'([m]) \to \gamma'([n])$ sends $t\in E([m])$ to the edge $e\in E([n])$ with $t\in E(f(e))$, if there is one, and to $*$ otherwise.

\begin{prop}\label{prop:classicsegal}
If $S\subseteq Y_{[1]}$ is the sieve on $[1]\in \Delta$ consisting of the constant maps $[n]\to[1]$ (equivalently, the non-surjective maps) then there is an isomorphism of functors $\Delta^\op\to\FinSet_*$
$$p : \gamma'\xto{\sim} Y_{[1]}/S$$ 
between Berger's Segal functor $\gamma'$ from Eqn. \eqref{eq:gammaprime} and the quotient construction above.
\end{prop}

\begin{proof}
In order to prove the proposition we construct a natural isomorphism 
$$p = \left\{ p : \gamma'([n]) \xto{\sim} (Y_{[1]}/S)([n]) \right\}_{[n]\in \ob(\Delta)}.$$
Since there is one edge $e_1\in \gamma([1])$ and, for each edge $e\in \gamma([n])$, there is a unique surjection $p_e:[n]\to [1]$ such that $p_e(e)=e_1$,
we may introduce maps $p : \gamma'([n])\to (Y_{[1]}/S)([n])$  where 
$$ p(e) := p_e \quad\quad\textnormal{ and }\quad\quad p(*) := *.$$
Each map $p$ is a bijection because
$(Y_{[1]}/S)([n]) = \left(\Hom_\Delta([n],[1])\backslash S([n])\right)_+$
 consists of the non-constant maps $[n]\to[1]$ and the basepoint,
by Construction \ref{construction}, and every surjective map $[n]\to [1]$ in the category $\Delta$ is equal to $p_e$ for exactly one $e\in\gamma([n])$.

To see that $p$ is natural, for each map $f:[n]\to [m]$ in $\Delta$, we check that the diagram below commutes, that is, $f^*\circ p = p \circ \gamma'(f)$.
\[\begin{tikzcd}
	{\gamma'([m])} & {(Y_{[1]}/S)([m])} \\
	{\gamma'([n])} & {(Y_{[1]}/S)([n])}
	\arrow["p", from=1-1, to=1-2]
	\arrow["{\gamma'(f)}"', from=1-1, to=2-1]
	\arrow["{f^*}", from=1-2, to=2-2]
	\arrow["p"', from=2-1, to=2-2]
\end{tikzcd}\]
On one hand, if $\gamma'(f)(t) = e$ then $p\gamma'(f)(t) = p_e$, and if $\gamma'(f)(t)=*$ then $p\gamma'(f)(t) = *$. On the other hand, $p(t) = p_t : [m] \to [1]$ and $f^*(p(t)) = p_t\circ f : [n] \xto{f} [m] \to [1]$, which is $p_e$ when $t\in E(f(e))$ and is constant, so equal to $*$ in the quotient, otherwise. The two sides agree.
\end{proof}

As noted above, any sieve $S\subseteq Y_{[1]}$ determines a Segal functor $(Y_{[1]}/S)^{op} : \Delta \to \G$. The theorem below shows that sieves of the form $S\subseteq Y_{[n]}$  are determined by  the images of the maps they contain.

\begin{thm}\label{thm:sievesondelta}
Let $\mathcal{I} := \{ S \subseteq Y_{[n]} : S \text{ is a sieve on } [n] \}$ and let 
    \begin{align*}
        \mathcal{S}:=\{S\subs \mathcal{P}([n])\backslash\{\emptyset\}: \left(A\in S \text{ and } \emptyset\neq B\subs A\right) \Rightarrow B\in S\}
    \end{align*}
    be the set of families of non-empty subsets of $[n] = \{0,1,\ldots,n\}$ which are closed under passing to non-empty subsets.  Then there are mutually inverse maps
    $$\Phi : \mathcal{S} \leftrightarrows \mathcal{I} : \Psi$$
which are determined by the assignments
\begin{align*}
  \Phi(S)([k]) :=\{f\in \Hom_{\Delta}([k],[n]):\im(f) \in S\} \text{ and } \\
  \Psi(I) := \{ S : S = \im(f) \text{ for some } f\in I([k]) \}.
  \end{align*}
\end{thm}

\begin{proof}
The proof consists of four steps. First we show that for $S\in\mathcal{S}$, $\Phi(S)\subseteq Y_{[n]}$ is a sieve on $[n]$. Then we check that for $I\in\mathcal{I}$, $\Psi(I)\in \mathcal{S}$. Lastly, we compute that $\Psi\Phi = 1_{\cal{S}}$ and $\Phi\Psi = 1_{\cal{I}}$.

{\em Step 1: If $S\in \mathcal{S}$ then $\Phi(S)$ is a sieve.}  We have $\Phi(S)([k]) \subseteq Y_{[n]}([k])$, and if $f : [k]\to [\ell]$ then the action of $Y_{[n]}$ on $f$ is precomposition $g\mapsto g\circ f$. This carries $\Phi(S)([\ell])$ into $\Phi(S)([k])$ because if $\im(g) \in S$ then $\emptyset\neq\im(g\circ f) \subseteq \im(g)$, so $\im(g\circ f)\in S$.

{\em Step 2: If $I\in \mathcal{I}$ then $\Psi(I)\in\mathcal{S}$.}  
    Suppose that $A\in\Psi(I)$ and $\emptyset\neq B\subs A$.  Since $A\in \Psi(I)$ there is an $f \in I([k])$, $f : [k]\to [n]$, such that $A=\im(f)$. If $j := \#f^{-1}(B)-1$ then there is a map $g : [j] \to [k]$ in $\Delta$ with $\im(g) = f^{-1}(B)$, so that $\im(f\circ g) = B$, and $f\circ g \in I([j])$ because $I$ is a sieve. So $B\in \Psi(I)$ and $\Psi(I)\in \mathcal{S}$.

{\em Step 3: $\Psi\Phi=1_{\mathcal{S}}$.}   For $S\in \cal{S}$, 
\begin{align*}
\Psi(\Phi(S)) &= \{ A : A = \im(f) \text{ for some } f\in \Phi(S)([k]) \} \\
&= \{ A : A = \im(f) \text{ for some } f: [k]\to [n] \text{ such that } \im(f) \in S \}\\
&= S
\end{align*}
where the last equality holds because every non-empty subset $A\subseteq[n]$ is the image of the unique order preserving injection $[\#A-1]\to[n]$ onto $A$.

{\em Step 4: $\Phi\Psi=1_{\mathcal{I}}$.}    For $I\in \cal{I}$,
    \begin{align*}
        \Phi(\Psi(I))([k])&= \{f\in \Hom_{\Delta}([k],[n]):\im(f) \in \Psi(I)\}\\
&= \{f\in \Hom_{\Delta}([k],[n]):\im(f) = \im(g) \text{ for some } g \in I([\ell]) \} \\
&= I([k]).
    \end{align*}
The inclusion $\supseteq$ in the last line is clear. For $\subseteq$,
suppose that $\im(f)=\im(g)$ with $g\in I([\ell])$. Then $h : [k]\to[\ell]$,
$h(i) := \min g^{-1}(f(i))$, is order preserving and $f = g\circ h$, so
$f\in I([k])$ because $I$ is a sieve.
\end{proof}

The following characterization of Berger's Segal functor $\gamma : \Delta
\to\G$ from Def.  \ref{def:segalmap} above follows from combining the
classification Thm. \ref{thm:sievesondelta} above with
Prop. \ref{prop:classicsegal}.

\begin{cor}\label{cor:berger}
    The sieve $S\subs Y_{[1]}$ consisting of the constant maps $[n]\to[1]$
    corresponds to Berger's Segal functor for $\Delta$, and $S$ is the
    largest proper sieve on $[1]$.
\end{cor}
\begin{proof}
The first statement is Prop. \ref{prop:classicsegal}. Using
Thm. \ref{thm:sievesondelta}, $\Psi(S) = \{ \{0\}, \{1\} \}$. The only
strictly larger family in $\mathcal{S}$ is $\Psi(S) \cup \{\{0,1\}\}$, which
corresponds to all of $Y_{[1]}$ and so is not proper.
  \end{proof}

\begin{rem}\label{rem:othersieves}
The other sieves on $[1]$ also determine Segal functors $\Delta\to\G$: the
empty sieve gives $[n]\mapsto\Hom_\Delta([n],[1])$, which labels all pairs
$i\leq j$ in $[n]$ rather than only the edges, and the sieves $\{\{0\}\}$
and $\{\{1\}\}$ give the two intermediate functors.
\end{rem}

\subsection{Segal functors for crossed simplicial groups}\label{sec:segalcss}

In this section we will apply some of our ideas to crossed simplicial groups $\DG$.  The categories
$\DG$  were introduced by Fiedorowicz and Loday \cite{FL} and Krasauskas \cite{Kras}. 
A crossed simplicial group is an extension of the simplex category
$\Delta$ by a groupoid in which the composition is required to be of a particularly simple form.
For us these categories constitute a concrete family of examples, at least some of which, appear in practice.

\begin{defn}{($\DG$)}\label{defn:css}
  A {\em crossed simplicial group} $\DG$ is a category equipped with a functor $i : \Delta \to \DG$  such that
  \begin{enumerate}
  \item $i$ is bijective on objects.
    \item Every map $f : [m]\to [n]$ in $\DG$ has a unique factorization of the form $f = \phi g$ where $\phi$ is in the image of $i$ and $g$ is an automorphism of $[m]$. 
    \end{enumerate}
  \end{defn}

If we set $\Gp_n := \Aut_{\DG}([n])$ then the collection $\Gp := \{\Gp_n\}_{n=0}^{\infty}$ forms a  simplicial set, see \cite[Lem. 1.3]{FL}.
A crossed simplicial group $\DG$ is an extension of the simplex category by this collection of groups.
The unique factorization property ensures that for each map $\phi : [m]\to [n]$ in $\Delta$ and $g\in\Gp_n$ there are maps 
$$\phi^* : \Gp_n\to \Gp_m \quad\quad\text{ and }\quad\quad g_* :
\Hom_{\DG}([m],[n]) \to \Hom_{\DG}([m],[n])$$ so that the composition
$g\phi$ can be rearranged to $g_*(\phi) \phi^*(g)$. In particular, since any
two maps can be written $\phi g$ and $\phi'g'$ their composition in $\DG$ can be described by 
\begin{equation}\label{eq:cssfact}
(\phi g) \circ_{\DG} (\phi'g') = (\phi\circ_\Delta g_*(\phi')) (\phi'^*(g)\circ_{\Gp_*} g').
\end{equation}
a composition in $\Delta$ and $\Gp_*$ respectively.

\begin{ex}
 When $\Gp_n = \{1\}$ for all $n\geq 0$ the associated crossed simplicial group $\DG = \Delta$ is the simplicial category.
\end{ex}

\newcommand{\ZZ}{\mathbb{Z}}
\newcommand{\inp}[1]{\ensuremath{\langle #1 \rangle}}
\newcommand{\cyc}{\Lambda}
\begin{ex}\label{ex:cyclic_category}
Connes' cyclic category $\cyc$ is a crossed simplicial group with $\Gp_n = \ZZ/(n+1)$. If the set $[n] = \{0,1,\ldots, n\}$ is viewed as the set of $(n+1)$st roots of unity in the unit circle $S^1 \subset \mathbb{C}$ then a map $f : [n]\to[m]$ in $\cyc$ is a homotopy class of increasing degree 1 map $f : (S^1,[n]) \to (S^1,[m])$.

A combinatorial definition is given using the categories $\<n\> = \{0<1<\cdots <n<0\}$ which are freely generated by the cyclic graphs below
    % https://q.uiver.app/#q=WzAsNCxbMCwxLCIwIl0sWzIsMSwibiJdLFsyLDAsIlxcZGRvdHMiXSxbMCwwLCIxIl0sWzEsMCwiIiwyLHsiY3VydmUiOi0yfV0sWzAsMywiIiwyLHsiY3VydmUiOi0yfV0sWzMsMiwiIiwyLHsiY3VydmUiOi0yfV0sWzIsMSwiIiwyLHsiY3VydmUiOi0yfV1d
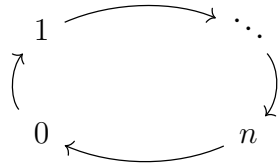
\[\begin{tikzcd}
	1 && \ddots \\
	0 && n
	\arrow[curve={height=-12pt}, from=1-1, to=1-3]
	\arrow[curve={height=-12pt}, from=1-3, to=2-3]
	\arrow[curve={height=-12pt}, from=2-1, to=1-1]
	\arrow[curve={height=-12pt}, from=2-3, to=2-1]
\end{tikzcd}\]
If $\omega_i:i\to i$ is the unique map of degree one then a functor $F:\<n\>\to \<m\>$ is called degree one when $F(\omega_i)=\omega_{F(i)}$ for all $i\in \<n\>$.  Connes' cyclic category $\Lambda$ is equivalent to the subcategory of $\cat{Cat}$ whose objects are $\<n\>$ and whose morphisms are the degree one functors \cite{DKcrossed}.
  \end{ex}

\begin{ex}
The paracyclic crossed simplicial group $\Lambda_\infty$ has structure groups $\Gp_n = \ZZ$.
If $\cyc_\infty$ is the category consisting of an object $\tilde{n} := \ZZ$ with the standard order for each non-negative integer $n$ then
$$\Hom_{\cyc_\infty}(\tilde{n},\tilde{m}) := \{ f : \ZZ\to\ZZ\,|\, f(l+m+1) = f(l) + n + 1, f \text{ order preserving} \}$$
Alternatively, if $\ZZ = \inp{t_{n+1}}$ then $\Delta\ZZ$ has a presentation with relations
  \begin{align*}
    t_{n+1}\d_i = \d_{i-1}t_n \quad\textnormal{ for }\quad 1\leq i \leq n  &\quad\quad\textnormal{ and }\quad\quad  t_{n+1}\d_0 = \d_n,\\
    t_{n+1}\s_i = \s_{i-1}t_{n+2} \quad\textnormal{ for }\quad 1 \leq i \leq n  &\quad\quad\textnormal{ and }\quad\quad  t_{n+1}\s_0 = \s_n t^2_{n+2}.
    \end{align*}
The cyclic category is the quotient $\cyc \cong \cyc_\infty/\inp{t_{n+1}^{n+1}=1_{[n]} : n\in\ZZ_{\geq 0}}$.
  \end{ex}

\begin{ex}\label{ex-ztwo}
  There is a crossed simplicial group $\Delta\ZZ/2$ with the same objects as $\Delta$, $\ob(\Delta \ZZ/2) := \{0<1<\cdots<n\}$ and set maps of the form
$$ f : [n]\to [m] \quad\quad\textnormal{ or }\quad\quad f^* : [n] \to [m]$$
so that
$$\Hom_{\Delta\ZZ/2}([n],[m]) = \Hom_{\Delta}([n],[m]) \sqcup \Hom_{\Delta}([n],[m])^*$$
such that $f$ preserves the order and $f^*$ reverses order. There is a special map $y_{n+1} := 1^* : [n]\to [n]$ 
generating a group $\ZZ/2$ which can be thought of as order reversing,
$$y_{n+1}\cdot\{0<1<2<\cdots<n\} := \{0>1>2>\cdots>n\} = \{n<n-1<\cdots<1<0\}.$$
On maps $\Hom_{\Delta}([n],[m])^* = \Hom_{\Delta}([n],[m])\cdot y_{n+1}$.
These generators $y_{n+1}$ satisfy the relations
  \begin{align*}
    y_{n+1}\d_i = \d_{n-i}y_n \quad\quad\textnormal{ and }\quad\quad  y_{n+1}\s_i = \s_{n-i}y_{n+2}
    \end{align*}
\end{ex}

If $\DG$ is a crossed simplicial group then let 
\begin{equation}\label{eq:znyonda}
Y^G_{[n]}([k]) := \Hom_{\DG}([k],[n])\quad\quad\text { where }\quad\quad Y^G_{[n]} : \DG^{\op}\to \FinSet 
\end{equation}
be the Yoneda functor  on $[n]$ in $\DG$ (compare to Eqn. \eqref{yonedaeq}). As in \S\ref{sec:chardefsegal}, by adding basepoints, Constr. \ref{construction} shows that any sieve $S \subseteq Y^G_{[n]}$ 
gives a functor $Y^G_{[n]}/S : \DG \to \FinSet_*$
so that $(Y^G_{[n]}/S)^{\op} : \DG\to \G$ is a candidate Segal functor. In the proposition below
classifies sieves on $[n]$ in $\DG$ in terms of sieves on $[n]$ in $\Delta$. Since the latter are classified by Thm. \ref{thm:sievesondelta}, this gives a classification of the sieves on $[n]$ in $\DG$.
%we classify sieves of $Y^G_{[n]}$ in terms of sieves on $Y_{[n]}$ since the latter is the content of Thm. \ref{thm:sievesondelta} this proposition gives a classification of sieves of $[n]$ in $\DG$.

\begin{prop}\label{prop:induce}
    Suppose that $S\subs Y_{[n]}$ is a sieve.  If $Y^G_{[n]}$ is the Yoneda functor in Eqn. \eqref{eq:znyonda} above
then  there is a sieve 
$S^{G}\subs Y^G_{[n]}$
which consisting of maps $\zeta : [k]\to [n] \in \DG$ which factor as
    $\zeta=\phi g$ with $g\in \Gp_{k}$ and $\phi \in S([k])$.
Moreover, every sieve on $[n]$ in $\DG$ arises in this way.
\end{prop}

\begin{proof}
First we prove that $S^G$ is a sieve. It suffices to show that if $\xi : [\ell] \to [k]$ is a map $\xi \in \DG$ then the image of $\xi^* : S^G([k]) \to Y^G([\ell])$ is contained in the subset $S^G([\ell])$.
Fix a map $\xi:[\ell]\to [k]$ in $\DG$ and let $\zeta\in S^{G}([k])$.  
The map $\xi$ factors as $\xi=\psi h$ and the map $\zeta$ factors as $\zeta = \phi g$. 
Now, as in Eqn. \eqref{eq:cssfact}, the pullback factors as 
$$\xi^*(\zeta) = \zeta\xi = (\phi\circ g_*\psi)\circ (\psi^*g\circ h).$$
Since $S$ is a sieve on $[n]$ implies that $\phi\circ g_*\psi\in S([\ell])$. So by definition, $\zeta\xi\in S^{G}([\ell])$.

Conversely, suppose that $T^G \subs Y^G_{[n]}$ is an arbitrary sieve.  Define
$T^\Delta$ to be those elements in $T^G$ which are morphisms in $\Delta$.  As
$i : \Delta\to \DG$ is faithful, $T^\Delta$ is a sieve on $[n]$ in
$\Delta$.

We claim that $(T^\Delta)^G\subs T^{G}$.  To see this, fix $f\in (T^\Delta)^G$.  Factor $f=\phi g$ with $g$ an automorphisms and $\phi$ in $\Delta$.  By definition, $\phi\in T^\Delta\subs T^G$.  Thus, since $T^G$ is a sieve, we have $f=\phi g\in T^G$.
On the other hand, if $f\in T^G$ and $f=\phi g$ is its factorization then
$\phi = fg^{-1}\in T^G$ because $T^G$ is a sieve. So $\phi\in T^{\Delta}$
and $f\in (T^\Delta)^G$. This shows $T^G\subseteq (T^\Delta)^G$. 
Since $T^G = (T^\Delta)^G$ all sieves $\DG$ arise in this way.
\end{proof}

\begin{rem}
If $S$ is a sieve on $[n]\in \Delta$ then $S^G \cong i_!S$, the left Kan extension of $S$ along $i^{\op}$. On the other hand, if $T^G$ is a sieve on $[n]\in \DG$ then $T^\Delta$ is not the restriction $i^*T^G$. 
\end{rem}

\begin{ex}\label{ex:segal-functors-from-sieves}
  Recall from Corollary \ref{cor:berger} that Berger's sieve is the largest proper sieve.
  \begin{enumerate}
  \item There is a Segal functor on $\Delta\ZZ/2$ crossed simplicial category of Ex. \ref{ex-ztwo}, $\g_{\Delta\mathbb{Z}/2} : \Delta\ZZ/2 \to \G$ which is defined by the quotient $\g_{\Delta\ZZ/2}^{op} := Y^{\ZZ/2}_{[1]}/B$ where $B$ consists of constant maps $\{0\}$ and $\{1\}$.
  \item There is a Segal functor on the cyclic category $\g_\Lambda : \Lambda \to \G$ which is defined by the quotient $\g_{\Lambda}^{op} := Y_{\inp{0}}/B$. In this case, there is only one vertex and $B=\emptyset$.
    \end{enumerate}
  \end{ex}

\begin{rem}
Example \ref{ex:segal-functors-from-sieves} showcases an interesting phenomenon.  Following Ex. \ref{ex-ztwo} $\Delta \Z/2$ consists of two directed graphs the Segal functor $\g_{\Delta\ZZ/2}$ in (1) above extends $\gamma:\Delta\to \Gamma$  by using these extra edges.
\end{rem}

\begin{rem}
The opposite of (2) agrees with a cyclic analogue of the Segal functor for $\Delta$ in Def. \ref{def:segalmap}. Let $E(\<n\>)$ be the edges of the graph that generates $\<n\>$ pictured above.  If $\phi\in\<n\>$ is a map then write $E(\phi)\subseteq E(\<n\>)$ for the set of elements whose composite is $\phi$. The functor $\gamma_{\Lambda}:\Lambda\to \Gamma$ is $\gamma_\Lambda(\<n\>):=E(\<n\>)$ on objects.  For a map $f:\<n\>\to \<m\>$ and $e\in E(\<n\>)$, set $\gamma_{\Lambda}(f)(e):=E(f(e))$.
\end{rem}

\begin{rem}
In a different direction, for any functor $i : \Delta\to C$ one could define the Segal functor $\gamma_C : C\to \G$ 
to be the left Kan extension $\gamma_C:=i_!\gamma$, when it exists.  
It appears that this also agrees with the Segal functor for $\Lambda$ in (2) above.
\end{rem}

Once Segal functors $\gG$ have been chosen, there is a Berger-Joyal duality for $n$-fold products of crossed simplicial groups.

\begin{cor}
  If $\DG_1, \DG_2,\ldots, \DG_n$ are crossed simplicial groups equipped with functors $\DG_i\to \Gamma$ then there is an isomorphism of categories
  $$(\DG_1 \otimes^\M \DG_2\otimes^{\M} \cdots \otimes^\M \DG_n)^{\op} \cong \DG_1^{\op}\otimes^{\Md} \DG_2^{\op} \otimes^{\Md} \cdots \otimes^{\Md} \DG_n^\op$$
  \end{cor}

The righthand side is the natural candidate for a category of $G$-disks.

\bibliographystyle{alpha}
\bibliography{Split1}

\end{document}